\documentclass[12pt]{article}

\usepackage{amsmath}
\usepackage{amsfonts}
\usepackage{amsthm}
\usepackage{tikz}
\newtheorem{theorem}{Theorem}[section]
\numberwithin{equation}{section}

\begin{document}
\title{Computing recurrence coefficients of multiple orthogonal polynomials\footnote{GF is supported by MNiSzW Iuventus Plus grant Nr 0124/IP3/2011/71, MH and WVA are supported by FWO research project G.0934.13, KU Leuven research grant OT/12/073 and the Belgian Interuniversity Attraction Poles Programme P7/18.}}
\author{Galina Filipuk\footnote{Faculty of Mathematics, Informatics and Mechanics, University of Warsaw, Poland. \newline
filipuk@mimuw.edu.pl},
\ Maciej Haneczok\footnotemark[3],\ Walter Van Assche\footnote{Department of Mathematics, KU Leuven, Belgium. \newline 
maciej.haneczok@wis.kuleuven.be; walter.vanassche@wis.kuleuven.be}} 
\date{}
\maketitle

\begin{abstract}
Multiple orthogonal polynomials satisfy a number of recurrence relations, in particular there is a $(r+2)$-term recurrence relation connecting the type II multiple orthogonal polynomials near the diagonal (the so-called step-line recurrence relation) and there is a system of $r$ recurrence relations connecting the nearest neighbors (the so-called nearest neighbor recurrence relations). In this paper we deal with two problems. First we show how one can obtain the nearest neighbor recurrence coefficients (and in particular the recurrence coefficients of the orthogonal polynomials for each of the defining measures) from the step-line recurrence coefficients. Secondly we show how one can compute the step-line recurrence coefficients from the recurrence coefficients of the orthogonal polynomials of each of the measures defining the multiple orthogonality.
\end{abstract}

\section{Introduction}

Multiple orthogonal polynomials are polynomials of one variable that satisfy orthogonality conditions with respect to $r>1$ positive measures.
In this paper we will only consider positive measures on the real line. Let $\vec{n}=(n_1,\ldots,n_r) \in \mathbb{N}^r$ be a multi-index and
$|\vec{n}| = n_1 + \cdots + n_r$ its length and let $\mu_1,\ldots,\mu_r$ be $r$ positive measures on the real line. There are two types of multiple orthogonal polynomials \cite[Chapter 23]{Ismail}, \cite{ABVA}, \cite{ACVA}, \cite{ECVA}. Type I multiple orthogonal polynomials are such that $(A_{\vec{n},1},\ldots,A_{\vec{n},r})$ is a vector of $r$ polynomials, with $\deg A_{\vec{n},j} \leq n_j-1$, satisfying 
\begin{equation}  \label{TypeIorth}
   \int x^k \sum_{j=1}^r A_{\vec{n},j}(x) \, d\mu_j(x) = 0, \qquad 0 \leq k \leq |\vec{n}|-2, 
\end{equation}
with normalization
\[   \int x^{|\vec{n}|-1} \sum_{j=1}^r A_{\vec{n},j}(x) \, d\mu_j(x) = 1.  \]
Type II multiple orthogonal polynomials are monic polynomials $P_{\vec{n}}$ of degree $|\vec{n}|$ for which
\begin{equation} \label{TypeIIorth} 
   \int x^k P_{\vec{n}}(x) \, d\mu_j(x) = 0, \qquad 0 \leq k \leq n_j-1 
\end{equation}
holds for $1 \leq j \leq r$. The existence (and unicity) is not guaranteed, but if the type I and type II multiple polynomials exist with the above normalization, then they are unique and then the multi-index $\vec{n}$ is said to be a normal index. The measures $(\mu_1,\ldots,\mu_r)$ are a normal system
if all the multi-indices are normal.

Multiple orthogonal polynomials satisfy a number of recurrence relations. Let
$p_{rn}(x) = P_{(n,n,\ldots,n)}(x)$, $p_{rn+1}(x) = P_{(n+1,n,n,\ldots,n)}(x)$, and in general
$p_{rn+j}(x) = P_{(n+1,\ldots,n+1,n\ldots,n)}(x)$ for $0 \leq j \leq r$, where $(n+1,\ldots,n+1,n\ldots,n)$ has $j$ times the component $n+1$ and
$r-j$ times the component $n$, i.e.,
\[  (n+1,\ldots,n+1,n\ldots,n) =  (n,n,\ldots,n) + \sum_{i=1}^j \vec{e}_i, \]
where $\vec{e}_i$ are the standard unit vectors in $\mathbb{N}^r$. Then the \textit{step-line polynomials} $p_{m}(x)$ satisfy the following
$(r+2)$-term recurrence relation
\begin{equation}  \label{stepr-rec}
    xp_{m}(x) = p_{m+1}(x) + \sum_{j=0}^r \beta_{m,j} p_{m-j}(x), 
\end{equation}
where $\beta_{m,j}$ are real recurrence coefficients and $p_0=1$ and $p_{-j}=0$ for $1 \leq j \leq r$. This recurrence relation corresponds
to the well known three term recurrence relation for orthogonal polynomials when $r=1$.
These step-line polynomials (the type II multiple orthogonal polynomials near the diagonal) are also known as $d$-orthogonal polynomials
(with $d=r$) and the orthogonality in \eqref{TypeIIorth} becomes 
\[    \int p_m(x) x^k\, d\mu_j(x) = 0, \qquad 0 \leq k \leq \lfloor (m-j)/r \rfloor. \]
See, e.g., \cite{Maroni}, \cite{DouakMaroni}, \cite{BenCheikhDouak}. This recurrence relation only connects multiple orthogonal polynomials
near the diagonal $(n,n,\ldots,n)$.
All multiple orthogonal polynomials (of a normal system) are related by a system of $r$ recurrence relations relating $P_{\vec{n}}$ with
its nearest neighbors $P_{\vec{n}+\vec{e}_k}$ and $P_{\vec{n}-\vec{e}_j}$. The system of \textit{nearest neighbor recurrence relations} 
is given by (see \cite{WVA})
\begin{equation}  \label{NNrec}
   x P_{\vec{n}}(x) = P_{\vec{n}+\vec{e}_k}(x) + b_{\vec{n},k} P_{\vec{n}}(x) + \sum_{j=1}^r a_{\vec{n},j} P_{\vec{n}-\vec{e}_j}(x), \qquad
    1 \leq k \leq r.
\end{equation}

In this paper we deal with two problems. For the first problem we assume that the recurrence coefficients $\beta_{m,j}$, $m \in \mathbb{N},\  0 \leq j \leq r$, in the step-line recurrence relation \eqref{stepr-rec} are given, and the goal is to find all the nearest neighbor recurrence coefficients 
$a_{\vec{n},j}, b_{\vec{n},j}$ with $\vec{n} \in \mathbb{N}^r$ and $1 \leq j \leq r$. In particular this would give all the recurrence coefficients
of the (monic) orthogonal polynomials for each of the measures $\mu_j$:
\[   \int P_n(x;\mu_j)P_m(x;\mu_j)\, d\mu_j(x) = \gamma_n^{-2}(\mu_j) \delta_{m,n}, \]
which satisfy the three term recurrence relation
\[   xP_n(x;\mu_j) = P_{n+1}(x;\mu_j) + b_n(\mu_j) P_n(x;\mu_j) + a_n^2(\mu_j) P_{n-1}(x;\mu_j). \]
Indeed, one has 
\[    b_n(\mu_j) = b_{n\vec{e}_j,j}, \quad   a_n^2(\mu_j) = a_{n\vec{e}_j,j}.  \]
In Section 2 we will show, for $r=2$, how one first can obtain the recurrence relation for the $r$ shifted step-line polynomials, i.e., the multiple orthogonal polynomials with a multi-index $k\vec{e}_j + \vec{n}$, with $k \in \mathbb{N}$ and $1\leq j \leq r$ fixed and $\vec{n}$ a multi-index on the step-line. Then, in a second step, we show how to obtain the nearest neighbor recurrence coefficients from the shifted step-line recurrence coefficients.
We explain the procedure for $r=2$ so that the reasoning is easy to follow and not obscured by the notation. 

The second problem is the inverse of the first problem: we assume that the recurrence coefficients $a_{n+1}(\mu_j),b_n(\mu_j)$ of the orthogonal polynomials of each of the measures $\mu_j$, $1 \leq j \leq r$, are given, and we show how one can compute all the nearest neighbor recurrence
coefficients and the step-line recurrence coefficients from this input. In Section 3 we will explain the case $r=2$ in detail. We will also show
how to find the nearest neighbor recurrence coefficients from the marginal recurrence coefficients for general $r$.

In order to find the recurrence coefficients  $\beta_{m,j}$, $m \in \mathbb{N},\  0 \leq j \leq r$ in the step-line recurrence relation, one may use
either the Jacobi-Perron algorithm or the vector QD-algorithm. The Jacobi-Perron algorithm generates a vector continued fraction and was introduced by Jacobi in 1868 \cite{Jacobi} and studied in detail by Perron in 1907 \cite{Perron}. A modern version of the Jacobi-Perron algorithm and its relevance for simultaneous rational approximation of functions can be found in \cite{Parus}. Vector continued fractions and the Jacobi-Perron algorithm are quite popular in number theory to produce simultaneous Diophantine approximations to several real numbers, see the monographs of Schweiger \cite{Schweiger1} \cite{Schweiger2}. Another way to obtain the step-line recurrence coefficients from the moments of the measures $\mu_1,\ldots,\mu_r$ is to use a generalization of the QD-algorithm proposed by Van Iseghem \cite{Iseghem}.
The classical QD-algorithm of Rutishauser \cite{Rutis} can be used to find the recurrence coefficients $a_{n+1}(\mu_j),b_n(\mu_j)$ of the orthogonal polynomials of each of the measures $\mu_j$, but one can also use the (modified) Chebyshev algorithm as described in \cite[\S 2.1.7]{Gautschi}.
In this paper we assume that the step-line recurrence coefficients are given (for the first problem) or we assume that the recurrence coefficients 
$a_{n+1}(\mu_j),b_n(\mu_j)$ are given for every measure $\mu_j$ with $1 \leq j \leq r$ (for the second problem).

\section{The recurrence coefficients along the step-line}

In this section we only consider multiple orthogonal polynomials with $r=2$. Hence the (monic) type II multiple orthogonal polynomials
$P_{n,m}$ depend on a multi-index $(n,m) \in \mathbb{N}^2$. 
Let $p_{2n}(x) = P_{n,n}(x)$ and $p_{2n+1}(x) = P_{n+1,n}(x)$, then the recurrence relation along the step-line is
\begin{equation}  \label{step-line}
  xp_n(x) = p_{n+1}(x) + \beta_n p_n(x) + \gamma_n p_{n-1}(x) + \delta_n p_{n-2}(x). 
\end{equation}
It is important to note that the step-line recurrence coefficients $(\beta_n)_{n \geq 0}, (\gamma_n)_{n\geq 1}$ and $(\delta_n)_{n \geq 2}$
do not determine the measures $\mu_1$ and $\mu_2$ in a unique way, even if we normalize the measures to be probability measures. The first measure
is determined uniquely as a probability measure, but for the second measure one can use any convex combination $\lambda \mu_1 + (1-\lambda) \mu_2$
because
\[   \lambda \int p_n(x) x^k\, d\mu_1(x) +  (1-\lambda) \int p_n(x) x^k d\mu_2(x) = 0, \qquad k \leq \lfloor \frac{n}{2} \rfloor -1, \]
since the first of these integrals vanishes for $k \leq \lfloor \frac{n-1}2 \rfloor $ and $\lfloor \frac{n}{2} \rfloor -1 \leq \lfloor \frac{n-1}2 \rfloor$, see \cite[Remark 2.2]{JC}.
This degree of freedom will
be reflected when we want to compute the recurrence coefficients of the orthogonal polynomials $p_n(x;\mu_1)$ and $p_n(x;\mu_2)$.
 
The nearest neighbor recurrence relations are
\begin{eqnarray}
   x P_{n,m}(x) &=& P_{n+1,m}(x) + c_{n,m} P_{n,m}(x) + a_{n,m} P_{n-1,m}(x) + b_{n,m} P_{n,m-1}(x), \label{NN1} \\
   x P_{n,m}(x) &=& P_{n,m+1}(x) + d_{n,m} P_{n,m}(x) + a_{n,m} P_{n-1,m}(x) + b_{n,m} P_{n,m-1}(x). \label{NN2}
\end{eqnarray}
As a consequence one has
\begin{equation}  \label{P-P}
    P_{n+1,m}(x) - P_{n,m+1}(x) = \kappa_{n,m} P_{n,m}(x), 
\end{equation}
where $\kappa_{n,m} = d_{n,m}-c_{n,m}$.

\subsection{From step-line to shifted step-line}

We introduce for $j \geq 0$ the polynomials 
\begin{equation}  \label{shiftj}
 p_{2n+j}^{(j,0)}(x) = P_{n+j,n}(x), \quad p_{2n+j+1}^{(j,0)}(x) = P_{n+j+1,n}(x), 
\end{equation}
and for $k \geq 0$
\begin{equation}  \label{shiftk}
 p_{2n+k}^{(0,k)}(x) = P_{n,n+k}(x), \quad p_{2n+k+1}^{(0,k)}(x) = P_{n+1,n+k}(x). 
\end{equation}
The polynomials $p_n^{(j,0)}$ are the multiple orthogonal polynomials on a shifted step-line with a shift $j$ in the direction of $\vec{e}_1 = (1,0)$. The polynomials $p_n^{(0,k)}$ are those on the shifted step-line with a shift $k$ in the direction of $\vec{e}_2 = (0,1)$.

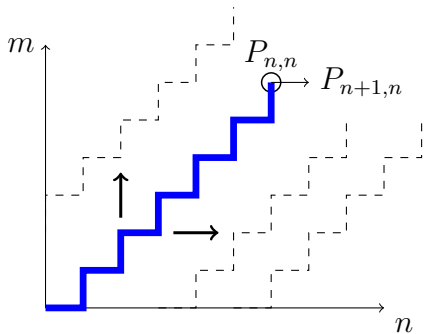
\begin{figure}[ht]
\centering
\begin{tikzpicture}[domain=0:2]
\draw[->] (0,0) -- (4.5,0) node[below right] {$n$};
\draw[->] (0,0) -- (0,3.5) node[left] {$m$};
\draw[line width=0.2mm] (3,3) circle (0.7ex) node[above]{$P_{n,n}$};
%---
\draw[line width=0.9mm, color=blue] (0,0) -- (0.5,0) -- (0.5,0.5) -- (1,0.5) -- (1,1) -- (1.5,1) -- (1.5,1.5)  -- (2,1.5) -- (2,2) -- (2.5,2) -- (2.5,2.5) -- (3,2.5) -- (3,3);
%---
\draw[->, line width=0.4mm] (1.7,1) -- (2.3,1);
\draw[-, color=black, dashed] (1.5,0) -- (2,0) -- (2,0.5) -- (2.5,0.5) -- (2.5,1) -- (3,1) -- (3,1.5) -- (3.5,1.5) -- (3.5,2) -- (4,2) -- (4,2.5);
\draw[-, color=black, dashed] (2.5,0) -- (3,0) -- (3,0.5) -- (3.5,0.5) -- (3.5,1) -- (4,1) -- (4,1.5) -- (4.5,1.5) -- (4.5,2) -- (5,2) -- (5,2.5);
%---
\draw[->, line width=0.4mm] (1,1.2) -- (1,1.8);
\draw[-, color=black, dashed] (0,1) -- (0,1.5) -- (0.5,1.5) -- (0.5,2) -- (1,2) -- (1,2.5) -- (1.5,2.5)  -- (1.5,3) -- (2,3) -- (2,3.5) -- 
(2.5,3.5) -- (2.5,4);
%---
\draw[->] (3,3) -- (3.5,3) node[right] {$P_{n+1,n}$};
\end{tikzpicture}
\caption{Step-line and shifted step-lines}
\label{fig:step-lines}
\end{figure}

These shifted step-line polynomials again satisfy a four term recurrence relation, which we denote by
\begin{equation}  \label{step-line+j}
  xp_n^{(j,0)}(x) = p_{n+1}^{(j,0)}(x) + \beta_n^{(j,0)} p_n^{(j,0)}(x) + \gamma_n^{(j,0)} p_{n-1}^{(j,0)}(x) + \delta_n^{(j,0)} p_{n-2}^{(j,0)}(x), 
\end{equation}
with initial conditions $\gamma _j^{(j,0)} = \delta _j^{(j,0)} = \delta _{j+1} ^{(j,0)}=0$ and in a similar way
\begin{equation}  \label{step-line+k}
  xp_n^{(0,k)}(x) = p_{n+1}^{(0,k)}(x) + \beta_n^{(0,k)} p_n^{(0,k)}(x) + \gamma_n^{(0,k)} p_{n-1}^{(0,k)}(x) + \delta_n^{(0,k)} p_{n-2}^{(0,k)}(x),
\end{equation}
with initial conditions $\gamma _k^{(0,k+1)} = \delta _k^{(0,k+1)} = \delta _{k+1} ^{(0,k+1)}=0$. 
With this notation we have $\beta_n^{(0,0)} = \beta_n$, $\gamma_n^{(0,0)}=\gamma_n$, and $\delta_n^{(0,0)}=\delta_n$. We introduce two more sequences
$(c_n^{(j,0)})_{n\geq 0}$ and $(c_n^{(0,k)})_{n \geq 0}$ by
\begin{equation}  \label{c(j,0)}
      P_{n+j+1,n}(x) - P_{n+j,n+1}(x) = c_n^{(j,0)} P_{n+j,n}(x), 
\end{equation}
and
\begin{equation}  \label{c(0,k)}
      P_{n+1,n+k}(x) - P_{n,n+k+1}(x) = c_n^{(0,k)} P_{n,n+k}(x), 
\end{equation}
so that \eqref{P-P} gives
\begin{equation}  \label{kappa}
 c_n^{(j,0)} = \kappa_{n+j,n} = d_{n+j,n}-c_{n+j,n}, \quad c_n^{(0,k)} = \kappa_{n,n+k} = d_{n,n+k}-c_{n,n+k}.  
\end{equation}

Our first result shows how one can obtain the recurrence coefficients $(\beta_n^{(j,0)},\gamma_n^{(j,0)},\delta_n^{(j,0)})$ from the recurrence coefficients along the step-line.

\begin{theorem}  \label{thm:(j,0)}
One has for all $j \geq 1$
\begin{eqnarray}  \label{2n+j}
   \beta_{2n+j}^{(j,0)} &=&  \beta_{2n+j}^{(j-1,0)} - c_n^{(j,0)}, \qquad n \geq 0,  \nonumber \\
   \gamma_{2n+j}^{(j,0)} &=& \gamma_{2n+j}^{(j-1,0)}, \qquad\qquad n \geq 1,  \\
   \delta_{2n+j}^{(j,0)} &=& \delta_{2n+j}^{(j-1,0)} - c_{n-1}^{(j,0)} \gamma_{2n+j}^{(j-1,0)}, \qquad n \geq 1. \nonumber
\end{eqnarray}
and
\begin{eqnarray}  \label{2n+j+1}
   \beta_{2n+j+1}^{(j,0)} &=&  \beta_{2n+j+1}^{(j-1,0)} + c_n^{(j,0)}, \qquad n \geq 0, \nonumber \\
   \gamma_{2n+j+1}^{(j,0)} &=& \gamma_{2n+j+1}^{(j-1,0)} + c_n^{(j,0)} (\beta_{2n+j}^{(j-1,0)} - \beta_{2n+j+1}^{(j,0)}), \qquad n \geq 0, \\
   \delta_{2n+j+1}^{(j,0)} &=& \delta_{2n+j+1}^{(j-1,0)} + c_n^{(j,0)} \gamma_{2n+j}^{(j-1,0)}, \qquad n \geq 0,  \nonumber 
\end{eqnarray}
where $(c_n^{(j,0)})_{n \geq 0}$ is the solution of the Riccati type difference equation
\begin{equation}   \label{cn(j,0)-Riccati}
   c_n^{(j,0)} = \frac{c_{n-1}^{(j,0)}\delta_{2n+j+1}^{(j-1,0)}}{\delta_{2n+j}^{(j-1,0)} - c_{n-1}^{(j,0)} \gamma_{2n+j}^{(j-1,0)}}, 
\end{equation}
with initial condition
\begin{equation}   \label{c0(j,0)}
    c_0^{(j,0)} = - \frac{\delta_{j+1}^{(j-1,0)}}{\gamma_j^{(j-1,0)}} .
\end{equation}
\end{theorem}

\begin{proof}
Take the recurrence relation  \eqref{step-line+j} with $n$ replaced by $2n+j$
\[ xp_{2n+j}^{(j,0)}(x) = p_{2n+j+1}^{(j,0)}(x) + \beta_{2n+j}^{(j,0)} p_{2n+j}^{(j,0)}(x) + \gamma_{2n+j}^{(j,0)} p_{2n+j-1}^{(j,0)}(x) 
+ \delta_{2n+j}^{(j,0)} p_{2n+j-2}^{(j,0)}(x), \]
use \eqref{shiftj} to find
\begin{equation}  \label{xp}
  xP_{n+j,n}(x) = P_{n+j+1,n}(x) + \beta_{2n+j}^{(j,0)} P_{n+j,n}(x) + \gamma_{2n+j}^{(j,0)} P_{n+j,n-1}(x) 
 + \delta_{2n+j}^{(j,0)} P_{n+j-1,n-1}(x).  
\end{equation}
Now use \eqref{c(j,0)} to replace $P_{n+j+1,n}(x)$ and (by changing $n$ to $n-1$) $P_{n+j,n-1}(x)$, to find
\begin{eqnarray}  \label{xp1}
 xP_{n+j,n}(x) &=& P_{n+j,n+1}(x) + \left(\beta_{2n+j}^{(j,0)} + c_n^{(j,0)}\right) P_{n+j,n}(x) +  \gamma_{2n+j}^{(j,0)} P_{n+j-1,n}(x) \nonumber \\
  & & +\ \left(\delta_{2n+j}^{(j,0)} + c_{n-1}^{(j,0)}\gamma_{2n+j}^{(j,0)}\right)  P_{n+j-1,n-1}(x).  
\end{eqnarray}
On the other hand, we take the recurrence relation   \eqref{step-line+j} for $j-1$
\[   xp_n^{(j-1,0)}(x) = p_{n+1}^{(j-1,0)}(x) + \beta_n^{(j-1,0)} p_n^{(j-1,0)}(x) + \gamma_n^{(j-1,0)} p_{n-1}^{(j-1,0)}(x) 
+ \delta_n^{(j-1,0)} p_{n-2}^{(j-1,0)}(x),  \]
replace $n$ by $2n+j$ and use  \eqref{shiftj} to find
\begin{eqnarray}  \label{xp2}
 xP_{n+j,n}(x) &=& P_{n+j,n+1}(x) + \beta_{2n+j}^{(j-1,0)} P_{n+j,n}(x) +  \gamma_{2n+j}^{(j-1,0)} P_{n+j-1,n}(x) \nonumber \\
   & & +\ \delta_{2n+j}^{(j-1,0)} P_{n+j-1,n-1}(x).  
\end{eqnarray}
Comparing \eqref{xp1} and \eqref{xp2} gives
\begin{multline*}
   \left(\beta_{2n+j}^{(j,0)} + c_n^{(j,0)}\right) P_{n+j,n}(x) +  \gamma_{2n+j}^{(j,0)} P_{n+j-1,n}(x) 
+ \left(\delta_{2n+j}^{(j,0)} + c_{n-1}^{(j,0)}\gamma_{2n+j}^{(j,0)}\right)  P_{n+j-1,n-1}(x) \\
  =  \beta_{2n+j}^{(j-1,0)} P_{n+j,n}(x) +  \gamma_{2n+j}^{(j-1,0)} P_{n+j-1,n}(x) 
+ \delta_{2n+j}^{(j-1,0)} P_{n+j-1,n-1}(x).
\end{multline*}
The polynomials $P_{n+j,n}, P_{n+j-1,n}, P_{n+j-1,n-1}$ are linearly independent (since they have degrees $2n+j, 2n+j-1$ and $2n+j-2$), hence
one finds
\begin{eqnarray*}  
  \beta_{2n+j}^{(j,0)} + c_n^{(j,0)} &=& \beta_{2n+j}^{(j-1,0)}, \\
  \gamma_{2n+j}^{(j,0)} &=& \gamma_{2n+j}^{(j-1,0)} , \\
  \delta_{2n+j}^{(j,0)} + c_{n-1}^{(j,0)}\gamma_{2n+j}^{(j,0)} & =& \delta_{2n+j}^{(j-1,0)},
\end{eqnarray*}
which gives the required relations \eqref{2n+j}.

In a similar way we start with the recurrence relation  \eqref{step-line+j} with $n$ replaced by $2n+j+1$
\[ xp_{2n+j+1}^{(j,0)}(x) = p_{2n+j+2}^{(j,0)}(x) + \beta_{2n+j+1}^{(j,0)} p_{2n+j+1}^{(j,0)}(x) + \gamma_{2n+j+1}^{(j,0)} p_{2n+j}^{(j,0)}(x) 
+ \delta_{2n+j+1}^{(j,0)} p_{2n+j-1}^{(j,0)}(x), \]
and use \eqref{shiftj} to find
\[  xP_{n+j+1,n}(x) = P_{n+j+1,n+1}(x) + \beta_{2n+j+1}^{(j,0)} P_{n+j+1,n}(x) + \gamma_{2n+j+1}^{(j,0)} P_{n+j,n}(x) 
 + \delta_{2n+j+1}^{(j,0)} P_{n+j,n-1}(x). \]
Use \eqref{c(j,0)} to replace $P_{n+j+1,n}(x)$ and $P_{n+j,n-1}(x)$, to find
\begin{eqnarray*}  
  x P_{n+j,n+1}(x) + c_n^{(j,0)} x P_{n+j,n}(x) &=& P_{n+j+1,n+1}(x) + \beta_{2n+j+1}^{(j,0)} P_{n+j,n+1}(x) \\
   & & +\ \left(\gamma_{2n+j+1}^{(j,0)} + c_n^{(j,0)} \beta_{2n+j+1}^{(j,0)}\right) P_{n+j,n}(x) \\
   & & +\  \delta_{2n+j+1}^{(j,0)} P_{n+j-1,n}(x)
   + c_{n-1}^{(j,0)} \delta_{2n+j+1}^{(j,0)} P_{n+j-1,n-1}(x) .
\end{eqnarray*}
Use \eqref{xp2} to replace $xP_{n+j,n}(x)$ to find
\begin{eqnarray}  \label{xp3}
x P_{n+j,n+1}(x) &=&  P_{n+j+1,n+1}(x) + \left(\beta_{2n+j+1}^{(j,0)} - c_n^{(j,0)} \right)  P_{n+j,n+1}(x) \nonumber \\
   & & +\ \left(\gamma_{2n+j+1}^{(j,0)} + c_n^{(j,0)} \beta_{2n+j+1}^{(j,0)} - c_n^{(j,0)} \beta_{2n+j}^{(j-1,0)} \right) P_{n+j,n}(x) \nonumber \\
   & & +\  \left( \delta_{2n+j+1}^{(j,0)} - c_n^{(j,0)} \gamma_{2n+j}^{(j-1,0)} \right) P_{n+j-1,n}(x) \nonumber \\
   & & +\ \left( c_{n-1}^{(j,0)} \delta_{2n+j+1}^{(j,0)} - c_n^{(j,0)} \delta_{2n+j}^{(j-1,0)} \right) P_{n+j-1,n-1}(x) .
\end{eqnarray}
On the other hand, we take the recurrence relation \eqref{step-line+j} for $j-1$ with $n$ replaced by $2n+j+1$ and use
\eqref{shiftj} to find
\begin{eqnarray}  \label{xp4}
  xP_{n+j,n+1}(x) &=& P_{n+j+1,n+1}(x) + \beta_{2n+j+1}^{(j-1,0)} P_{n+j,n+1}(x) \nonumber \\
                  & & +\ \gamma_{2n+j+1}^{(j-1,0)} P_{n+j,n}(x) + \delta_{2n+j+1}^{(j-1,0)} P_{n+j-1,n}(x).
\end{eqnarray}
Comparing \eqref{xp3} and \eqref{xp4} then gives
\begin{multline*} 
 \left(\beta_{2n+j+1}^{(j,0)} - c_n^{(j,0)} \right)  P_{n+j,n+1}(x) + \left(\gamma_{2n+j+1}^{(j,0)} + c_n^{(j,0)} \beta_{2n+j+1}^{(j,0)} - c_n^{(j,0)} \beta_{2n+j}^{(j-1,0)} \right) P_{n+j,n}(x) \\  
  +  \left( \delta_{2n+j+1}^{(j,0)} - c_n^{(j,0)} \gamma_{2n+j}^{(j-1,0)} \right) P_{n+j-1,n}(x)
 + \left( c_{n-1}^{(j,0)} \delta_{2n+j+1}^{(j,0)} - c_n^{(j,0)} \delta_{2n+j}^{(j-1,0)} \right) P_{n+j-1,n-1}(x) \\
  = \beta_{2n+j+1}^{(j-1,0)} P_{n+j,n+1}(x) + \gamma_{2n+j+1}^{(j-1,0)} P_{n+j,n}(x) + \delta_{2n+j+1}^{(j-1,0)} P_{n+j-1,n}(x) .
\end{multline*}
The four polynomials $P_{n+j,n+1}, P_{n+j,n}, P_{n+j-1,n}, P_{n+j-1,n-1}$ are linearly independent, hence one finds
\begin{eqnarray}  \label{extra}
  \beta_{2n+j+1}^{(j,0)} - c_n^{(j,0)} & = & \beta_{2n+j+1}^{(j-1,0)}, \nonumber \\
  \gamma_{2n+j+1}^{(j,0)} + c_n^{(j,0)} \beta_{2n+j+1}^{(j,0)} - c_n^{(j,0)} \beta_{2n+j}^{(j-1,0)} & = & \gamma_{2n+j+1}^{(j-1,0)}, \nonumber \\
  \delta_{2n+j+1}^{(j,0)} - c_n^{(j,0)} \gamma_{2n+j}^{(j-1,0)} & =& \delta_{2n+j+1}^{(j-1,0)}, \nonumber \\
  c_{n-1}^{(j,0)} \delta_{2n+j+1}^{(j,0)} - c_n^{(j,0)} \delta_{2n+j}^{(j-1,0)} & = & 0. 
\end{eqnarray}
The first three relations give \eqref{2n+j+1} and the last equation gives \eqref{cn(j,0)-Riccati} if we replace $\delta_{2n+j+1}^{(j,0)}$
by the third equation in \eqref{2n+j+1}. For $n=0$ the relation \eqref{xp3} becomes
\begin{eqnarray*}  
x P_{j,1}(x) &=&  P_{j+1,1}(x) + \left(\beta_{j+1}^{(j,0)} - c_0^{(j,0)} \right)  P_{j,1}(x) \nonumber \\
   & & +\ \left(\gamma_{j+1}^{(j,0)} + c_0^{(j,0)} \beta_{j+1}^{(j,0)} - c_0^{(j,0)} \beta_{j}^{(j-1,0)} \right) P_{j,0}(x) \nonumber \\
   & &  -\  c_0^{(j,0)} \gamma_{j}^{(j-1,0)}  P_{j-1,0}(x) ,
\end{eqnarray*}
and if we compare the coefficient of $P_{j-1,0}(x)$ with the corresponding coefficient in \eqref{xp4} when $n=0$, then
\[   - c_0^{(j,0)} \gamma_{j}^{(j-1,0)} = \delta_{j+1}^{(j-1,0)}, \]
which gives \eqref{c0(j,0)}.
\end{proof}

The Riccati equation \eqref{cn(j,0)-Riccati} can be solved explicitly if all the step-line coefficients at shift $j-1$ are known.
The substitution $c_n^{(j,0)} = 1/d_n^{(j,0)}$ gives
\[    d_n^{(j,0)} = \frac{d_{n-1}^{(j,0)} \delta_{2n+j}^{(j-1,0)} - \gamma_{2n+j}^{(j-1,0)}}{\delta_{2n+j+1}^{(j-1,0)}}, \]
which is a first order linear recurrence relation. Its solution is
\begin{equation}  \label{d(j,0)}
d_n ^{(j,0)}  = \left( \sum_{i=1}^{n} \frac{\gamma _{2i+j} ^{(j-1,0)}} { \delta _{2i+j+1} ^{(j-1,0)} } \prod _{k=1}^{i} \frac {\delta _{2k+j+1}^{(j-1,0)}} {\delta _{2k+j}^{(j-1,0)}} + d_0^{(j,0)} \right) \prod _{k=1}^{n} \frac {\delta _{2k+j}^{(j-1,0)}}{\delta _{2k+j+1}^{(j-1,0)}} .
\end{equation}
However, this is only useful if one has explicit expressions for the step-line coefficients. 

An algorithm for computing the recurrence coefficients for the shifted step-line is as follows. Assume that all the recurrence coefficients
$\beta_n^{(j-1,0)},\gamma_n^{(j-1,0)}, \delta_n^{(j-1,0)}$ are known, then one first computes the auxiliary sequence $(c_n^{(j,0)})_{n \geq 0}$
recursively by using \eqref{cn(j,0)-Riccati}. Once this is done, one uses the relations \eqref{2n+j} and \eqref{2n+j+1} to get the
recurrence coefficients for the shifted step-line with shift $j$. The following Maple procedure computes
$(\beta_n^{(j,0)})_{j \leq n \leq 2N-j+1}, (\gamma_n^{(j,0)})_{j+1 \leq n \leq 2N-j+1}$ and $(\delta_n^{(j,0)})_{j+2 \leq n \leq 2N-j+1}$ for $1 \leq j \leq J$. 

\begin{verbatim}
shift_j:=proc(N,J) 
  local n,j; 
  for n from 0 to 2*N+J+1 do 
    b[n](0):=beta0(n); 
    g[n](0):=gamma0(n); 
    d[n](0):=delta0(n); 
  end do; 
  for j from 1 to J do 
    g[j](j):=0; 
    d[j](j):=0; 
    d[j+1](j):=0; 
    c[0](j):=-d[j+1](j-1)/g[j](j-1); 
    b[j](j):=b[j](j-1)-c[0](j); 
    b[j+1](j):=b[j+1](j-1)+c[0](j);  
    g[j+1](j):=g[j+1](j-1)+c[0](j)*(b[j](j-1)-b[j+1](j)); 
    for n from 1 to N-j do 
      c[n](j):=c[n-1](j)*d[2*n+j+1](j-1)/(d[2*n+j](j-1)-c[n-1](j)*g[2*n+j](j-1)); 
      b[2*n+j](j):=b[2*n+j](j-1)-c[n](j); 
      g[2*n+j](j):=g[2*n+j](j-1); 
      d[2*n+j](j):=d[2*n+j](j-1)-c[n-1](j)*g[2*n+j](j-1); 
      b[2*n+j+1](j):=b[2*n+j+1](j-1)+c[n](j); 
      g[2*n+j+1](j):=g[2*n+j+1](j-1)+c[n](j)*(b[2*n+j](j-1)-b[2*n+j+1](j)); 
      d[2*n+j+1](j):=d[2*n+j+1](j-1)+c[n](j)*g[2*n+j](j-1); 
    end do; 
  end do; 
end proc;
\end{verbatim}

\noindent\textbf{Remark:}
Observe that \eqref{cn(j,0)-Riccati} and the third equation in \eqref{2n+j} give
\[   c_n^{(j,0)} =c_{n-1}^{(j,0)} \frac{\delta_{2n+j+1}^{(j-1,0)}}{\delta_{2n+j}^{(j,0)}},  \]
and \eqref{extra} gives
\[   c_n^{(j,0)} =c_{n-1}^{(j,0)} \frac{\delta_{2n+j+1}^{(j,0)}}{\delta_{2n+j}^{(j-1,0)}}.  \]
Comparing both expressions gives
\[   \delta_{2n+j+1}^{(j,0)} = \frac{\delta_{2n+j+1}^{(j-1,0)}\delta_{2n+j}^{(j-1,0)}}{\delta_{2n+j}^{(j,0)}}.  \]
Hence one may therefore replace lines 16--24 in the Maple procedure by

\begin{verbatim}
    for n from 1 to N-j do 
      d[2*n+j](j):=d[2*n+j](j-1)-c[n-1](j)*g[2*n+j](j-1);
      d[2*n+j+1](j):=d[2*n+j+1](j-1)*d[2*n+j](j-1)/d[2*n+j](j);        
      c[n](j):=c[n-1](j)*d[2*n+j+1](j-1)/(d[2*n+j](j); 
      b[2*n+j](j):=b[2*n+j](j-1)-c[n](j); 
      g[2*n+j](j):=g[2*n+j](j-1); 
      b[2*n+j+1](j):=b[2*n+j+1](j-1)+c[n](j); 
      g[2*n+j+1](j):=g[2*n+j+1](j-1)+c[n](j)*(b[2*n+j](j-1)-b[2*n+j+1](j)); 
    end do; 
\end{verbatim}

The formulas in Theorem \ref{thm:(j,0)} and the computations in the algorithm hold provided $\delta_{2n+j}^{(j,0)} \neq 0$. 
If $\delta_{2n+j}^{(j,0)} = 0$ then we cannot compute $c_n^{(j,0)}$ and this quantity is needed in most of the other formulas
for the $j$-shifted step-line. The condition $\delta_{2n+j}^{(j,0)} \neq 0$ for all $j \geq 1$ and $n \geq 1$ is a sufficient condition
that implies that all the required recurrence coefficients can be computed and hence implies that the multi-indices $(n+k,n)$ are
normal.
\medskip

There is a similar result for the recurrence coefficients $(\beta_n^{(0,k)},\gamma_n^{(0,k)},\delta_n^{(0,k)})$.

\begin{theorem}  \label{thm:(0,k)}
One has for all $k \geq 0$
\begin{eqnarray*}
   \beta_{2n+k}^{(0,k+1)} &=&  \beta_{2n+k}^{(0,k)} + c_n^{(0,k)}, \qquad n \geq 0, \\
   \gamma_{2n+k}^{(0,k+1)} &=& \gamma_{2n+k}^{(0,k)}, \qquad\qquad n \geq 1,  \\
   \delta_{2n+k}^{(0,k+1)} &=& \delta_{2n+k}^{(0,k)} + c_{n-1}^{(0,k)} \gamma_{2n+k}^{(0,k)}, \qquad n \geq 1.
\end{eqnarray*}
and
\begin{eqnarray*}
   \beta_{2n+k+1}^{(0,k+1)} &=&  \beta_{2n+k+1}^{(0,k)} - c_n^{(0,k)}, \qquad n \geq 0, \\
   \gamma_{2n+k+1}^{(0,k+1)} &=& \gamma_{2n+k+1}^{(0,k)} - c_n^{(0,k)} (\beta_{2n+k}^{(0,k)} - \beta_{2n+k+1}^{(0,k+1)}), \qquad n \geq 0, \\
   \delta_{2n+k+1}^{(0,k+1)} &=& \delta_{2n+k+1}^{(0,k)} - c_n^{(0,k)} \gamma_{2n+k}^{(0,k)}, \qquad n \geq 1, 
\end{eqnarray*}
where $(c_n^{(0,k)})_{n \geq 0}$ is the solution of the Riccati type difference equation
\begin{equation}   \label{cn(0,k)-Riccati}
   c_n^{(0,k)} = \frac{c_{n-1}^{(0,k)}\delta_{2n+k+1}^{(0,k)}}{\delta_{2n+k}^{(0,k)} + c_{n-1}^{(0,k)} \gamma_{2n+k}^{(0,k)}}, 
\end{equation}
with initial condition
\begin{equation}   \label{c0(0,k)}
    c_0^{(0,k)} =  \frac{\delta_{k+1}^{(0,k)}}{\gamma_k^{(0,k)}} , \qquad k \geq 1, 
\end{equation}
and for $k=0$ the $c_0^{(0,0)}$ is a free parameter.
\end{theorem}
\begin{proof}
The proof is very similar to the proof of Theorem \ref{thm:(j,0)}, but one uses the recurrence relation \eqref{step-line+k} and
the relations \eqref{shiftk} and \eqref{c(0,k)}.
\end{proof}

An important difference is that one also needs $c_0^{(0,0)}$ which one can find by taking $n=0$ and $k=0$ in the relation
for $\beta_{2n+k}^{(0,k+1)}$, giving
\[  c_0^{(0,0)} = \beta_0^{(0,1)} - \beta_0^{(0,0)}. \]
Here $\beta_0^{(0,0)} = \beta_0$ is known as the first of the step-line recurrence coefficients, but $\beta_0^{(0,1)}$ can not
be obtained in terms of the step-line recurrence coefficients. Recall that the step-line recurrence coefficients do not determine the measures $\mu_1$ and $\mu_2$ but only determine $\mu_1$ and for the second measure any convex combination of $\mu_1$ and $\mu_2$ is possible. This degree of freedom
is reflected in $c_0^{(0,0)}$ being a free parameter. 

Again, the Riccati equation \eqref{cn(0,k)-Riccati} can be solved explicitly. The substitution $c_n^{(0,k)} = 1/d_n^{(0,k)}$ gives
\[  d_n^{(0,k)} = \frac{d_{n-1}^{(0,k)} \delta_{2n+k}^{(0,k)} + \gamma_{2n+k}^{(0,k)}}{\delta_{2n+k+1}^{(0,k)}}, \]
and this first order linear recurrence has the following expression as solution
\begin{equation}  \label{d(0,k)}
d_n ^{(0,k)}  = \left( \sum_{i=1}^{n} \frac{\gamma _{2i+k} ^{(0,k)}} { \delta _{2i+k+1} ^{(0,k)} } \prod _{l=1}^{i} \frac {\delta _{2l+k+1}^{(0,k)}} {\delta _{2l+k}^{(0,k)}} + d_0^{(0,k)} \right) \prod _{l=1}^{n} \frac {\delta _{2l+k}^{(0,k)}}{\delta _{2l+k+1}^{(0,k)}} .
\end{equation}

The following Maple procedure computes $(\beta_n^{(0,k+1)})_{k \leq n \leq 2N-k+1}, (\gamma_n^{(0,k+1)})_{k+1 \leq n \leq 2N-k+1}$ and
$(\delta_n^{(0,k+1)})_{k+2 \leq n \leq 2N-k+1}$ for $0 \leq k \leq K$. It requires as extra input the first two moments of the measure $\mu_2$, i.e.,
$m(1,2)=m_1(\mu_2)$ and $m(0,2)=m_0(\mu_2)$. 

\begin{verbatim}
shift_k:=proc(N,K) 
  local n,k; 
  for n from 0 to 2*N+K+1 do 
    b[n](0) := beta0(n); 
    g[n](0) := gamma0(n); 
    d[n](0) := delta0(n); 
  end do; 
  c[0](0) := m(1,2)/m(0,2)-b[0](0); 
  for k from 0 to K do 
    g[k](k+1):=0; 
    d[k](k+1):=0; 
    d[k+1](k+1):=0; 
    b[k](k+1):=b[k](k)+c[0](k); 
    b[k+1](k+1):=b[k+1](k)-c[0](k);  
    g[k+1](k+1):=g[k+1](k)-c[0](k)*(b[k](k)-b[k+1](k+1)); 
    for n from 1 to N-k do 
      c[n](k):=c[n-1](k)*d[2*n+k+1](k)/(d[2*n+k](k)+c[n-1](k)*g[2*n+k](k)); 
      b[2*n+k](k+1):=b[2*n+k](k)+c[n](k); 
      g[2*n+k](k+1):=g[2*n+k](k); 
      d[2*n+k](k+1):=d[2*n+k](k)+c[n-1](k)*g[2*n+k](k); 
      b[2*n+k+1](k+1):=b[2*n+k+1](k)-c[n](k); 
      g[2*n+k+1](k+1):=g[2*n+k+1](k)-c[n](k)*(b[2*n+k](k)-b[2*n+k+1](k+1)); 
      d[2*n+k+1](k+1):=d[2*n+k+1](k)-c[n](k)*g[2*n+k](k); 
    end do; 
    c[0](k+1):=d[k+2](k+1)/g[k+1](k+1); 
  end do; 
end proc; 
\end{verbatim}

\noindent\textbf{Remark:}
One has
\[   c_n^{(0,k)} = c_{n-1}^{(0,k)} \frac{\delta_{2n+k+1}^{(0,k)}}{\delta_{2n+k}^{(0,k+1)}}, \]
and in a way similar as before one has
\[   \delta_{2n+k+1}^{(0,k+1)} = \frac{\delta_{2n+k+1}^{(0,k)}\delta_{2n+k}^{(0,k)}}{\delta_{2n+k}^{(0,k+1)}}. \]
Hence one can replace the lines 16--24 in the above procedure by

\begin{verbatim}
    for n from 1 to N-k do 
      d[2*n+k](k+1):=d[2*n+k](k)+c[n-1](k)*g[2*n+k](k);
      d[2*n+k+1](k+1):=d[2*n+k+1](k)*d[2*n+k](k)/d[2*n+k](k+1);
      c[n](k):=c[n-1](k)*d[2*n+k+1](k)/d[2*n+k](k+1); 
      b[2*n+k](k+1):=b[2*n+k](k)+c[n](k); 
      g[2*n+k](k+1):=g[2*n+k](k); 
      b[2*n+k+1](k+1):=b[2*n+k+1](k)-c[n](k); 
      g[2*n+k+1](k+1):=g[2*n+k+1](k)-c[n](k)*(b[2*n+k](k)-b[2*n+k+1](k+1));   
    end do; 
\end{verbatim}

A sufficient condition for normality of the multi-indices $(n,n+k)$ is now that $\delta_{2n+k}^{(0,k+1)} \neq 0$ for all $n \geq 1$ and $k \geq 0$.

\subsection{From shifted step-line to nearest neighbor recurrence coefficients}
Our next step is to find the nearest neighbor recurrence coefficients $a_{n,m},b_{n,m},c_{n,m},d_{n,m}$ in \eqref{NN1}--\eqref{NN2}
from the recurrence coefficients on the shifted step-lines. 

\begin{theorem}  \label{thm:NN}
The coefficients of the nearest neighbor recurrence relations \eqref{NN1}--\eqref{NN2} are for $n\geq 0$ and $j \geq 1$ given by
\begin{eqnarray*}  \label{down}
  c_{n+j,n} &=& \beta_{2n+j}^{(j,0)}, \\
  d_{n+j,n} &=& c_{n+j,n} + c_n^{(j,0)}, \\
  a_{n+j,n} &=& - \frac{\delta_{2n+j+1}^{(j-1,0)}}{c_n^{(j,0)}}, \\
  b_{n+j,n} &=& \gamma_{2n+j}^{(j,0)} - a_{n+j,n} .
\end{eqnarray*}
and for $n\geq 1$ and $k \geq 0$
\begin{eqnarray*}
   c_{n,n+k} &=& \beta_{2n+k}^{(0,k)}, \\
   d_{n,n+k} &=& c_{n,n+k} + c_n^{(0,k)}, \\
   a_{n,n+k} &=& - \frac{\delta_{2n+k}^{(0,k)}}{c_{n-1}^{(0,k)}}, \\
   b_{n,n+k} &=& \gamma_{2n+k}^{(0,k)} - a_{n,n+k}.
\end{eqnarray*}
The initial coefficients are $a_{0,0}=b_{0,0}=0$, $c_{0,0} = \beta_0$ and $d_{0,0}$ is a free parameter.
\end{theorem}

Note that if $b_0(\mu_1)=m_1(\mu_1)/m_0(\mu_1)$ and $b_0(\mu_2)=m_1(\mu_2)/m_0(\mu_2)$ are the first recurrence coefficients
of the orthogonal polynomials for $\mu_1$ and $\mu_2$ respectively, then $c_{0,0} = b_0(\mu_1)$ and $d_{0,0}=b_0(\mu_2)$.
Hence the particular choice $d_{0,0}=m_1(\mu_2)/m_0(\mu_2)$ gives the nearest neighbor recurrence coefficients of the
multiple orthogonal polynomials for the measures $\mu_1$ and $\mu_2$.
Another choice of $d_{0,0}$ still gives $\mu_1$ as the first measure, but a linear combination of $\mu_1$ and $\mu_2$ for the second measure.

\begin{proof}
We start from \eqref{xp} and replace the polynomial $P_{n+j-1,n-1}$ by using \eqref{c(j,0)} to find
\begin{multline*}
   xP_{n+j,n}(x) = P_{n+j+1,n}(x) + \beta_{2n+j}^{(j,0)} P_{n+j,n}(x) + \gamma_{2n+j}^{(j,0)} P_{n+j,n-1}(x)  \\
 +\ \frac{\delta_{2n+j}^{(j,0)}}{c_{n-1}^{(j,0)}}  \left( P_{n+j,n-1}(x)-P_{n+j-1,n}(x) \right). 
\end{multline*}
Now we compare this with the recurrence relation \eqref{NN1} with $n$ replaced by $n+j$ and $m=n$, then
\begin{eqnarray*}
       c_{n+j,n} &=& \beta_{2n+j}^{(j,0)}, \\
       b_{n+j,n} &=& \gamma_{2n+j}^{(j,0)} + \frac{\delta_{2n+j}^{(j,0)}}{c_{n-1}^{(j,0)}}, \\
       a_{n+j,n} &=& - \frac{\delta_{2n+j}^{(j,0)}}{c_{n-1}^{(j,0)}}.
\end{eqnarray*}
Use the last equation in (\ref{extra}) to obtain 
\[ a_{n+j,n} = - \frac{\delta_{2n+j+1} ^{(j-1,0)} }{ c_n ^{(j,0)} }, \]
so that the formula is valid for all $n\geq 0$. For $d_{n+j,n}$ we use the relation \eqref{kappa} to find
\[  d_{n+j,n} = c_{n+j,n} + c_n^{(j,0)}. \]
 This gives the first part of the theorem.   

For the second part of the theorem we use \eqref{step-line+k} with $n$ replaced by $2n+k$ and by \eqref{shiftk} this gives
\begin{equation}  \label{xp0}
  xP_{n,n+k}(x) = P_{n+1,n+k}(x) + \beta_{2n+k}^{(0,k)} P_{n,n+k}(x) + \gamma_{2n+k}^{(0,k)} P_{n,n+k-1}(x) 
+ \delta_{2n+k}^{(0,k)} P_{n-1,n+k-1}(x). 
\end{equation}
Replace the polynomial $P_{n-1,n+k-1}$ by using \eqref{c(0,k)} (but with $n$ replaced by $n-1$) to find
\begin{multline*}
  xP_{n,n+k}(x) = P_{n+1,n+k}(x) + \beta_{2n+k}^{(0,k)} P_{n,n+k}(x) + \gamma_{2n+k}^{(0,k)} P_{n,n+k-1}(x) \\
  +\   \frac{\delta_{2n+k}^{(0,k)}}{c_{n-1}^{(0,k)}} \left( P_{n,n+k-1}(x) - P_{n-1,n+k}(x) \right).
\end{multline*}
Comparing with \eqref{NN1} with $m=n+k$ then gives 
\begin{eqnarray*}
        c_{n,n+k} &=& \beta_{2n+k}^{(0,k)}, \\
        a_{n,n+k} &=& - \frac{\delta_{2n+k}^{(0,k)}}{c_{n-1}^{(0,k)}}, \\
        b_{n,n+k} &=& \gamma_{2n+k}^{(0,k)} + \frac{\delta_{2n+k}^{(0,k)}}{c_{n-1}^{(0,k)}} .
\end{eqnarray*}
For $d_{n,n+k}$ we use \eqref{kappa} to find
\[   d_{n,n+k} = c_{n,n+k} + c_n^{(0,k)}. \]
This gives the second part of the theorem. 
\end{proof}

In Maple one can compute the nearest neighbor recurrence coefficients using the following procedure:

\begin{verbatim}
nncoef:=proc(n,m) 
   local j,k; 
   if m<n then j:=n-m; 
             shift_j(n,j);  
             c(n,m):= b[2*m+j](j); 
             d(n,m):= c(n,m)+c[m](j); 
             a(n,m):= -d[2*m+j+1](j-1)/c[m](j); 
             b(n,m):= g[2*m+j](j)-a(n,m); 
          else k:=m-n; 
             shift_k(m,k); 
             c(n,m):= b[2*n+k](k); 
             d(n,m):= c(n,m)+c[n](k); 
             a(n,m):= -d[2*n+k](k)/c[n-1](k); 
             b(n,m):= g[2*n+k](k)-a(n,m); 
   end if; 
   Vector([c(n,m), d(n,m), a(n,m), b(n,m)]); 
end proc;
\end{verbatim}

\subsection{The recurrence coefficients of the marginal measures}
Now that we know the nearest neighbor recurrence coefficients, we can find the recurrence coefficients of the orthogonal polynomials
$P_n(x;\mu_1)$ for the measure $\mu_1$ and $P_n(x;\mu_2)$ for the measure $\mu_2$. The recurrence relations for these monic orthogonal
polynomials are
\[   xP_n(x;\mu_i) = P_{n+1}(x;\mu_i) + b_n(\mu_i) P_n (x;\mu_i) + a_n^2(\mu_i) P_{n-1}(x;\mu_i). \]
Observe that $P_n(x;\mu_1) = P_{n,0}(x)$ and $P_m(x;\mu_2) = P_{0,m}(x)$, so if we compare with the nearest neighbor recurrence relations
\eqref{NN1} and  \eqref{NN2} we find
\[    b_j(\mu_1) = c_{j,0}, \quad  a_j^2(\mu_1) = a_{j,0}, \]
and
\[   b_k(\mu_2) = d_{0,k},  \quad a_k^{2}(\mu_2) = b_{0,k}.  \]
If we use Theorem \ref{thm:NN} then this gives
\begin{equation}  \label{recmu1}
    b_j(\mu_1) = \beta_j^{(j,0)}, \quad  a_j^2(\mu_1) = \gamma_j^{(j-1,0)}, 
\end{equation}
where we used \eqref{c0(j,0)} to simplify the expression for $a_{j,0}$. 
The recurrence coefficients for $\mu_2$ can be obtained more easily from \eqref{xp0} with $n=0$, which gives
\[  xP_{0,k}(x) = P_{1,k}(x) + \beta_k^{(0,k)} P_{0,k}(x) + \gamma_k^{(0,k)} P_{0,k-1}(x), \]
and if we replace $P_{1,k}$ by using \eqref{c(0,k)} with $n=0$, then
\[   xP_{0,k}(x) = P_{0,k+1}(x) + \left( c_0^{(0,k)} + \beta_k^{(0,k)} \right) P_{0,k}(x) + \gamma_k^{(0,k)} P_{0,k-1}(x), \]
so that
\begin{equation}  \label{recmu2}
    b_k(\mu_2) = \beta_k^{(0,k)} + c_0^{(0,k)}, \quad
      a_k^2(\mu_2) = \gamma_k^{(0,k)} .  
\end{equation}

Observe that these results  allow us to find the recurrence coefficients of the orthogonal polynomials for the measures $\mu_1$ and $\mu_2$ if the
recurrence coefficients of the step-line multiple orthogonal polynomials are known. There are examples of multiple orthogonal polynomials for which the recurrence coefficients of the marginal orthogonal polynomials are not known. Such a situation occurs for instance in the case of multiple orthogonal polynomials associated to the modified Bessel functions of the first and second kind. In 1990 A.P. Prudnikov posed an open problem to find the orthogonal polynomials for the modified Bessel functions of the second kind $K_\nu(2\sqrt{x})$ on $[0,\infty)$ (see \cite[Problem 9 on pp. 239--241]{WVA-open}). It turned out that in this case it is more natural to consider multiple orthogonal polynomials for a pair of modified Bessel functions of the second kind $K_\nu(2\sqrt{x}),K_{\nu+1}(2\sqrt{x})$. This was shown by Van Assche and Yakubovich in \cite{WVA-Yaku} (see also \cite{BenCheikhDouak0}) and later for the modified Bessel functions of the first kind by Coussement and Van Assche in \cite{EC2} (see also \cite{Douak}). 
Our algorithm allows us to find the recurrence coefficients of those polynomials, even though we are not able to find explicit expressions for them.
If we start from the step-line recurrence coefficients given in \cite[Thm.~4]{WVA-Yaku}
\begin{eqnarray*}
  \beta_n &=& (n+\alpha+1)(3n+\alpha+2\nu) - (\alpha+1)(\nu-1), \\   
  \gamma_n &=& n(n+\alpha)(n+\alpha+\nu)(3n+2\alpha+\nu), \\
  \delta_n &=& n(n-1)(n+\alpha)(n+\alpha-1)(n+\alpha+\nu)(n+\alpha+\nu-1), 
\end{eqnarray*} 
for the multiple orthogonal polynomials with 
\[  d\mu_1(x)=x^{\alpha+\nu/2} K_\nu(2\sqrt{x}), \quad d\mu_2(x) = x^{\alpha+(\nu+1)/2} K_{\nu+1}(2\sqrt{x}), \]
and put $\alpha=0$ and $\nu=0$, then Table \ref{K0} gives the results of our algorithm for the recurrence coefficients 
$(a_{n},b_n)$ of the weight $K_0(2\sqrt{x})$.
    
\begin{table}[h]\footnotesize  
\centering
\begin{tabular}{|r|c|c|}
\hline
$n$ & $a_n$ & $b_n$  \\
\hline
0 & --  & 1 \\
1 & 1.7320508075688772935 & 9.6666666666666666667 \\
2 & 8.5374989832437982487 & 28.186991869918699187 \\ 
3 & 20.265386777687130909 & 56.571895845674401834 \\
4 & 36.925214834648582674 & 94.823932737801348717 \\
5 & 58.518554562959399225 & 142.94410230778264607 \\
6 & 85.045955898223602580 & 200.93289913274452209 \\
7 & 116.50767686120789662 & 268.79060407933245800 \\
8 & 152.90385976282648737 & 346.51739199614374938 \\
9 & 194.23459164836084172 & 434.11337913848760712 \\
10 & 240.49992974325090503 & 531.57864673346522330 \\
\hline 
\end{tabular}
\caption{Recurrence coefficients of orthogonal polynomials for $K_0(2\sqrt{x})$}
\label{K0}
\end{table}
Note that the values for the coefficients $a_n$ presented in Table \ref{K0} are the same as the ones that were computed in \cite{WVA-open} 
%by Van Assche
using the moments.

\section{From marginal to  nearest neighbor}

In the previous section we started from the step-line recurrence coefficients and we showed how to find the nearest neighbor
recurrence coefficients and in particular the recurrence coefficients of the orthogonal polynomials for the marginal measures
$\mu_1$ and $\mu_2$. In this section we will investigate the inverse problem: suppose the recurrence coefficients
$(a_{n+1}^2(\mu_i),b_n(\mu_i))_{n \geq 0}$ are given for $i \in \{1,2\}$. How can one find all the nearest neighbor recurrence
coefficients $(a_{n,m}, b_{n,m}, c_{n,m}, d_{n,m})$ and the step-line recurrence coefficients $(\beta_n, \gamma_n,\delta_n)$?

\subsection{The nearest neighbor recurrence coefficients}

The nearest neighbor recurrence coefficients satisfy a system of non-linear partial difference equations, as was noted in \cite{WVA}.
We will briefly show how to find these partial difference equations. The nearest neighbor recurrence relations \eqref{NN1} and \eqref{NN2} can be written in a matrix form as
\[      Y_{n+1,m} = R_1(n,m) Y_{n,m}, \quad Y_{n,m+1} = R_2(n,m) Y_{n,m}, \]
where
\[   Y_{n,m} =   \begin{pmatrix} P_{n,m}(x) \\ P_{n-1,m}(x) \\ P_{n,m-1}(x) \end{pmatrix}  \]
and $R_1$ and $R_2$ are the two transfer matrices
\[    R_1(n,m) = \begin{pmatrix} x-c_{n,m} & -a_{n,m} & -b_{n,m} \\
                                 1 & 0 & 0 \\
                                 1 & 0 & d_{n,m-1}-c_{n,m-1}  \end{pmatrix},  \]
and
\[    R_2(n,m) = \begin{pmatrix} x-d_{n,m} & -a_{n,m} & -b_{n,m} \\
                                 1 & c_{n-1,m}-d_{n-1,m} & 0 \\
                                 1 & 0 & 0  \end{pmatrix}.  \]
Now there are two ways of finding $Y_{n+1,m+1}$ from $Y_{n,m}$ using these transfer matrices: one way is to first compute $Y_{n+1,m}$ and then to increase $m$ by one, which gives
\[    Y_{n+1,m+1} = R_2(n+1,m)R_1(n,m)Y_{n,m}. \]
Another way is to first compute $Y_{n,m+1}$ and then to increase $n$ by one, giving
\[    Y_{n+1,m+1} = R_1(n,m+1) R_2(n,m) Y_{n,m}.  \]
Comparing both expressions gives the matrix relation
\[     R_2(n+1,m)R_1(n,m) =  R_1(n,m+1) R_2(n,m).  \]
If one computes the entries of this matrix identity, then one finds the following partial difference relations (see also \cite[Thm.~3.1]{WVA}):
\begin{eqnarray}
   d_{n+1,m}-d_{n,m} &=& c_{n,m+1}-c_{n,m} , \label{cd} \\
   a_{n+1,m} + b_{n+1,m} - (a_{n,m+1} + b_{n,m+1}) &=& \det \begin{pmatrix} d_{n+1,m} & d_{n,m} \\ c_{n,m+1} & c_{n,m} \end{pmatrix}, \label{a+b} \\
   \frac{a_{n,m+1}}{a_{n,m}} &=& \frac{c_{n,m}-d_{n,m}}{c_{n-1,m}-d_{n-1,m}}, \label{acd} \\
   \frac{b_{n+1,m}}{b_{n,m}} &=& \frac{c_{n,m}-d_{n,m}}{c_{n,m-1}-d_{n,m-1}}.  \label{bcd}
\end{eqnarray}

We will show that these partial difference equations with boundary conditions 
\begin{equation}  \label{BC1}
   a_{n,0} = a_n^2(\mu_1), \quad b_{n,0} = 0, \quad c_{n,0} = b_n(\mu_1), \qquad n \geq 0,  
\end{equation}
and
\begin{equation}  \label{BC2}
   a_{0,m} = 0, \quad b_{0,m} = a_m^2(\mu_2), \quad d_{0,m} = b_m(\mu_2), \qquad m \geq 0,
\end{equation}
where $a_n^2(\mu_i),b_n(\mu_i)$ are the recurrence coefficients of the monic orthogonal polynomials for the measure $\mu_i$ $(i=1,2)$ (with
$a_0^2(\mu_1)=a_0^2(\mu_2)=0$),
can be solved recursively to find the nearest neighbor recurrence coefficients for the multiple orthogonal polynomials with the measures $(\mu_1,\mu_2)$.

\begin{theorem}  \label{thm:inverse}
Suppose $(a_n^2(\mu_i))_{n \geq 1}$ and $(b_n(\mu_i))_{n \geq 0}$ are the recurrence coefficients of the monic orthogonal polynomials
$P_n(x;\mu_i)$ for the measure $\mu_i$, i.e.,
\[    xP_n(x;\mu_i) = P_{n+1}(x;\mu_i) + b_n(\mu_i) P_n(x;\mu_i) + a_n^2(\mu_i) P_{n-1}(x;\mu_i), \qquad n \geq 0, \]
with $P_0(x;\mu_i)=1$ and $P_{-1}(x;\mu_i)=0$. Then the nearest neighbor recurrence coefficients for the type II multiple orthogonal polynomials
can be computed recursively by \eqref{cd}--\eqref{bcd}, using the boundary conditions \eqref{BC1}--\eqref{BC2}, provided $c_{n,m} \neq d_{n,m}$
for all $n,m \geq 0$.
\end{theorem}
  
\begin{proof}
We use induction on $k$, where $k=n+m$ is the length of the multi-index $(n,m)$ and show how to compute $a_{n,m},b_{n,m},c_{n,m},d_{n,m}$
when the recurrence coefficients are known for multi-indices of length less than $k$.
For $k=0$ we have that $a_{0,0}=b_{0,0}=0$ (these appear as coefficients of $P_{-1,0}$ and $P_{0,-1}$ and hence are not needed) and
$c_{0,0}=b_0(\mu_1)$, $d_{0,0}=b_0(\mu_2)$. Hence the case $k=0$ is settled. For $k=1$ we already have
$a_{1,0}=a_1^2(\mu_1),c_{1,0}=b_1(\mu_1), b_{0,1}=a_1^2(\mu_2),d_{0,1}=b_1(\mu_2)$, and we also defined $a_{0,1}=b_{1,0}=0$ (these
appear as coefficients for $P_{-1,1}$ and $P_{1,-1}$ and are not needed). This leaves only to determine $c_{0,1}$ and $d_{1,0}$.
If we use \eqref{cd} and \eqref{a+b} for $n=m=0$, then this gives the system of equations
\begin{eqnarray*}
      d_{1,0}-c_{0,1} &=& d_{0,0}-c_{0,0}, \\
      d_{1,0}c_{0,0} - c_{0,1}d_{0,0} &=& a_{1,0}+b_{1,0}-a_{0,1}-b_{0,1}.
\end{eqnarray*}
This is a linear system of two equations for $c_{0,1}$ and $d_{1,0}$. The determinant of the system is $d_{0,0}-c_{0,0}$ and hence this system has a unique solution whenever $b_0(\mu_1) \neq b_0(\mu_2)$.

\begin{figure}[ht]
\centering
\begin{tikzpicture}[domain=0:2]
\draw[->, color=black] (0,0) -- (3.5,0) node[below right] {$n$};
\draw[->, color=black] (0,0) -- (0,3.5) node[left] {$m$};
%--
\draw[line width=0.5mm, color=blue, dashed] (0,0.5) -- (0.5,0);
\draw[line width=0.5mm, color=blue, dashed] (0,1) -- (1,0);
\draw[line width=0.5mm, color=blue, dashed] (0,1.5) -- (1.5,0);
\draw[->, line width=0.4mm] (1,1) -- (1.3,1.3);
%\draw[] (1,1)  node[right] { \rotatebox[origin=c]{260}{$\ddots$}};
\draw[line width=0.5mm, color=blue, dashed] (0,3) -- (3,0);
%--
\draw[] (-0.3,-0.3)  node[right] {$0$};
\draw[] (0.2,-0.3)  node[right] {$1$};
\draw[] (0.7,-0.3)  node[right] {$2$};
\draw[] (1.2,-0.3)  node[right] {$3$};
\draw[] (1.7,-0.3)  node[right] {$\ldots $};
\end{tikzpicture}
\caption{Moving along the lines $m+n=k$}
\label{fig:moving}
\end{figure}
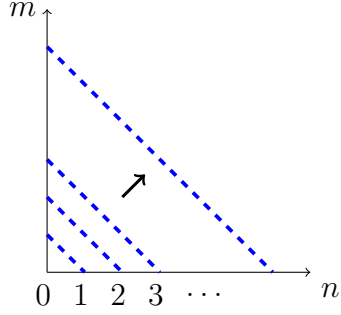

Suppose next that we know all the nearest neighbor recurrence coefficients for multi-indices $(n,m)$ of length $ \leq k$. 
From \eqref{acd} we then find
\begin{equation*}  
    a_{n,m+1} = a_{n,m} \frac{c_{n,m}-d_{n,m}}{c_{n-1,m}-d_{n-1,m}}, 
\end{equation*}
and from \eqref{bcd} we find  
\begin{equation*} 
    b_{n+1,m} = b_{n,m} \frac{c_{n,m}-d_{n,m}}{c_{n,m-1}-d_{n,m-1}}. 
\end{equation*}
If we replace $n$ by $\ell$ and $m$ by $k-\ell$, then this gives
\begin{equation}  \label{anm1}
  a_{\ell,k-\ell+1} = a_{\ell,k-\ell} \frac{c_{\ell,k-\ell}-d_{\ell,k-\ell}}{c_{\ell-1,k-\ell}-d_{\ell-1,k-\ell}}, \qquad 1 \leq \ell \leq k. 
\end{equation}
For $\ell=0$ and $\ell=k+1$ we use the boundary conditions 
\[ a_{0,k+1} = 0, \quad a_{k+1,0} = a_{k+1}^2(\mu_1). \]
In a similar way
\begin{equation}   \label{bnm1}
  b_{\ell+1,k-\ell} = b_{\ell,k-\ell} \frac{c_{\ell,k-\ell}-d_{\ell,k-\ell}}{c_{\ell,k-\ell-1}-d_{\ell,k-\ell-1}}, \qquad 0 \leq \ell \leq k-1. 
\end{equation}
For $\ell = -1$ and $\ell=k$ we use the boundary conditions
\[  b_{0,k+1} = a_{k+1}^2(\mu_2), \quad b_{k+1,0} = 0.  \]
The expressions on the right of \eqref{anm1}--\eqref{bnm1} contain coefficients of multi-indices of length $k$ and $k-1$ and hence they are known.
If we use \eqref{cd} and \eqref{a+b}, then
\begin{eqnarray*}
     d_{n+1,m} - c_{n,m+1} &=& d_{n,m} - c_{n,m} , \\
     d_{n+1,m}c_{n,m} - c_{n,m+1} d_{n,m} &=& a_{n+1,m}+b_{n+1,m}-a_{n,m+1}-b_{n,m+1}, 
\end{eqnarray*}
which is a linear system for $c_{n,m+1}$ and $d_{n+1,m}$ with determinant $d_{n,m}-c_{n,m}$. This system has a unique solution
whenever $c_{n,m}\neq d_{n,m}$. This solution is
\[   c_{n,m+1} = c_{n,m} + \frac{a_{n+1,m}+b_{n+1,m}-a_{n,m+1}-b_{n,m+1}}{c_{n,m}-d_{n,m}}, \]
and
\[   d_{n+1,m} = d_{n,m} + \frac{a_{n+1,m}+b_{n+1,m}-a_{n,m+1}-b_{n,m+1}}{c_{n,m}-d_{n,m}}.  \]
Replacing $n$ by $\ell$ and $m$ by $k-\ell$ then gives
\begin{equation} \label{cnm1}
    c_{\ell,k-\ell+1} = c_{\ell,k-\ell} + \frac{a_{\ell+1,k-\ell}+b_{\ell+1,k-\ell}-a_{\ell,k-\ell+1}-b_{\ell,k-\ell+1}}
   {c_{\ell,k-\ell}-d_{\ell,k-\ell}},  \qquad 0 \leq \ell \leq k, 
\end{equation}
and for $\ell=k+1$ we use the boundary condition
\[     c_{k+1,0} = b_{k+1}(\mu_1). \]
Similarly we have 
\begin{equation} \label{dnm1}
    d_{\ell+1,k-\ell} = d_{\ell,k-\ell} + \frac{a_{\ell+1,k-\ell}+b_{\ell+1,k-\ell}-a_{\ell,k-\ell+1}-b_{\ell,k-\ell+1}}
   {c_{\ell,k-\ell}-d_{\ell,k-\ell}},   \qquad 0 \leq \ell \leq k, 
\end{equation}
and for $\ell=-1$ we use the boundary condition
\[      d_{0,k+1} = b_{k+1}(\mu_2).  \]
\end{proof}

We can implement this in Maple using the following procedure which computes $a_{n,m},$ $b_{n,m},c_{n,m},d_{n,m}$ for $n+m \leq N+M$. It requires
the input $a1(n) = a_n^2(\mu_1)$ and $b1(n)=b_n(\mu_1)$ for the first measure $\mu_1$ and $a2(n) = a_n^2(\mu_2)$ and $b2(n)=b_n(\mu_2)$
for the second measure $\mu_2$, for $0 \leq n \leq N+M$, where we set $a_0^2(\mu_1)=0=a_0^2(\mu_2)$. In particular it gives the
coefficients $c_{N,M},d_{N,M},a_{N,M},b_{N,M}$ which will be given in the output explicitly.

\begin{verbatim}
IP:=proc(N,M) 
 local n,m,k; 
 for n from 0 to N+M do 
  c(n,0):=b1(n); 
  a(n,0):=a1(n); 
  b(n,0):=0; 
 end do; 
 for m from 0 to N+M do 
  d(0,m):=b2(m); 
  a(0,m):=0; 
  b(0,m):=a2(m); 
 end do; 
 for n from 1 to N+M do 
  for k from 1 to n-1 do 
    a(k,n-k):=a(k,n-k-1)*(c(k,n-k-1)-d(k,n-k-1))/(c(k-1,n-k-1)-d(k-1,n-k-1)); 
    b(k,n-k):=b(k-1,n-k)*(c(k-1,n-k)-d(k-1,n-k))/(c(k-1,n-k-1)-d(k-1,n-k-1)); 
  end do; 
  for k from 1 to n do 
   c(n-k,k):=c(n-k,k-1) 
       +(a(n-k+1,k-1)+b(n-k+1,k-1)-a(n-k,k)-b(n-k,k))/(c(n-k,k-1)-d(n-k,k-1)); 
  end do; 
  for k from 1 to n do 
   d(k,n-k):=c(k-1,n-k+1)-c(k-1,n-k)+d(k-1,n-k); 
  end do; 
 end do; 
Vector([c(N,M),d(N,M),a(N,M),b(N,M)]); 
end proc;  
\end{verbatim}

\subsection{The step-line recurrence coefficients}
Next we will show how to compute the recurrence coefficients in the step-line recurrence \eqref{step-line} if one knows the nearest neighbor
recurrence coefficients.

\begin{theorem}  \label{thm:stepline}
Suppose that the nearest neighbor recurrence coefficients are given. Then the step-line recurrence coefficients in \eqref{step-line}
are given by
\begin{eqnarray*}
     \beta_{2n} &=& c_{n,n},  \qquad n \geq 0,\\
     \gamma_{2n} &=& a_{n,n} + b_{n,n}, \qquad n \geq 1,\\
     \delta_{2n} &=& a_{n,n} ( c_{n-1,n-1} - d_{n-1,n-1}), \qquad n \geq 1,
\end{eqnarray*}
and
\begin{eqnarray*}
     \beta_{2n+1} &=& d_{n+1,n}, \qquad n \geq 0, \\
     \gamma_{2n+1} &=& a_{n+1,n} + b_{n+1,n}, \qquad n \geq 0, \\
     \delta_{2n+1} &=& b_{n+1,n} (d_{n,n-1}-c_{n,n-1}), \qquad n \geq 1.
\end{eqnarray*}
\end{theorem}

\begin{proof}
From the nearest neighbor recurrence relation \eqref{NN1} we find
\[   xP_{n,n}(x) = P_{n+1,n}(x) + c_{n,n} P_{n,n}(x) + a_{n,n} P_{n-1,n}(x) + b_{n,n} P_{n,n-1}(x). \]
Use \eqref{P-P} to replace $P_{n-1,n}$ to find
\begin{multline*}   xP_{n,n}(x) = P_{n+1,n}(x) + c_{n,n} P_{n,n}(x) 
    + (a_{n,n}+b_{n,n}) P_{n,n-1}(x)  \\
+ a_{n,n} (c_{n-1,n-1}-d_{n-1,n-1})P_{n-1,n-1}(x).  
\end{multline*}
If we compare this with \eqref{step-line} with $n$ replaced by $2n$, then we find the relations for the even recurrence coefficients.
The proof for the odd recurrence coefficients is similar: use \eqref{NN2} to find
\[   xP_{n+1,n}(x) = P_{n+1,n+1}(x) + d_{n+1,n} P_{n+1,n}(x) + a_{n+1,n} P_{n,n}(x) + b_{n+1,n} P_{n+1,n-1}(x), \]
and replace $P_{n+1,n-1}$ using \eqref{P-P}, giving
\begin{multline*}
   xP_{n+1,n}(x) = P_{n+1,n+1}(x) + d_{n+1,n} P_{n+1,n}(x) + (a_{n+1,n} + b_{n+1,n})P_{n,n}(x) \\
+ b_{n+1,n} (d_{n,n-1}-c_{n,n-1}) P_{n,n-1}(x). 
\end{multline*}
Comparing with the recurrence relation \eqref{step-line} with $n$ replaced by $2n+1$ gives the required result. 
\end{proof}

\subsection{The nearest neighbor coefficients for general $r$}

For general $r$ there are more partial difference equations for the nearest neighbor recurrence coefficients. The nearest neighbor recurrence relations \eqref{NNrec} can be written as
\[    Y_{\vec{n} + \vec{e}_k} = R_{k}(\vec{n}) Y_{\vec{n}}, \qquad 1 \leq k \leq r, \]
where
\[   Y_{\vec{n}} = \begin{pmatrix}
                   P_{\vec{n}}(x) \\ P_{\vec{n}-\vec{e}_1}(x) \\ \vdots \\ P_{\vec{n}-\vec{e}_r}(x)
                  \end{pmatrix}  \]
and $R_{k}(\vec{n})$ are $(r+1)\times(r+1)$ transfer matrices given by
\[   R_{k}(\vec{n}) = \begin{pmatrix}
                      x - b_{\vec{n},k} & - a_{\vec{n},1} & \cdots & -a_{\vec{n},k} & \cdots & -a_{\vec{n},r} \\
                      1 & b_{\vec{n}-\vec e_1,1}-b_{\vec{n}-\vec e_1,k} & \cdots & 0 & \cdots & 0 \\
                      \vdots &   & \ddots &  & & \vdots \\
                      1 & 0 & \cdots & 0 & \cdots & 0 \\
                      \vdots &   & & & \ddots & \vdots\\
                      1 & 0 & \cdots & 0 & \cdots & b_{\vec{n}-\vec e_r,r}-b_{\vec{n}-\vec e_r,k}  
                      \end{pmatrix}.  \]
Expressing that $Y_{\vec{n}+\vec{e}_i+\vec{e}_j}$ can be computed in two ways when $i \neq j$ is done by
\[    R_i(\vec{n}+\vec{e}_j)R_j(\vec{n}) = R_j(\vec{n}+\vec{e}_i)R_i(\vec{n}), \] 
and this gives the following partial difference relations \cite[Thm.~3.2]{WVA}: for all $1 \leq i \neq j \leq r$  one has
\begin{eqnarray}
   b_{\vec{n}+\vec{e}_i,j}-b_{\vec{n},j} &=& b_{\vec{n}+\vec{e}_j,i}-b_{\vec{n},i},  \label{PD1} \\
   \sum_{k=1}^r a_{\vec{n}+\vec{e}_j,k} - \sum_{k=1}^r a_{\vec{n}+\vec{e}_i,k} &=& \det \begin{pmatrix}
                               b_{\vec{n}+\vec{e}_j,i} & b_{\vec{n},i} \\
                               b_{\vec{n}+\vec{e}_i,j} & b_{\vec{n},j} \end{pmatrix}, \label{PD2}  \\
   \frac{a_{\vec{n}+\vec{e}_j,i}}{a_{\vec{n},i}} &=&   \frac{b_{\vec{n},j}-b_{\vec{n},i}}{b_{\vec{n}-\vec{e}_i,j}-b_{\vec{n}-\vec{e}_i,i}}.
   \label{PD3}
\end{eqnarray}

\begin{theorem}  
Suppose the recurrence coefficients $(a_n^2(\mu_i))_{n\geq 1}$ and $(b_n(\mu_i))_{n \geq 0}$ of the monic orthogonal polynomials for the measure
$\mu_i$ are known $(1 \leq i \leq r)$.
Then the nearest neighbor recurrence coefficients $a_{\vec{n},j}, b_{\vec{n},j}$ $(1 \leq j \leq r)$ can be computed from
\eqref{PD1}--\eqref{PD3} and the boundary conditions
\[   a_{n\vec{e}_j,j} = a_n^2(\mu_j), \quad b_{n\vec{e}_j,j} = b_n(\mu_j), \qquad n \geq 0,\ 1 \leq j \leq r, \]
where $a_0^2(\mu_j)=0$, and
\[  a_{n\vec{e}_i,j} = 0, \qquad n \geq 0,\ i \neq j, \]
provided that $b_{\vec{n},i} \neq b_{\vec{n},j}$ for all multi-indices $\vec{n} \in \mathbb{N}^r$ and $1 \leq i \neq j \leq r$.
\end{theorem}

\begin{proof}
We use induction on the length $N=|\vec{n}|$ of the multi-index $\vec{n}$. 
For $|\vec{n}|=0$ we see that $a_{\vec{0},j}=0$ and $b_{\vec{0},j}=b_0(\mu_j)$ for $1 \leq j \leq r$. If $|\vec{n}|=1$ then
$\vec{n}=\vec{e}_i$ for some $i$ with $1 \leq i \leq r$. Therefore $a_{\vec{n},j}= a_{\vec{e}_i,j} = 0$ whenever $i \neq j$ and
$a_{\vec{e}_i,i}=a_1^2(\mu_i)$. Furthermore $b_{\vec{n},i} = b_{\vec{e}_i,i} = b_1(\mu_i)$. Using \eqref{PD1} we also find
\[   b_{\vec{e}_i,j} = b_{\vec{e}_j,i} + b_0(\mu_j) - b_0(\mu_i),  \]
and  \eqref{PD2} gives
\[   b_{\vec{e}_j,i}b_0(\mu_j) - b_{\vec{e}_i,j}b_0(\mu_i) =  \sum_{k=1}^r a_{\vec{e}_j,k} - \sum_{k=1}^r a_{\vec{e}_i,k}. \]
Solving this linear system gives for $i \neq j$
\[    b_{\vec{e}_j,i} = b_0(\mu_i) + \frac{\sum_{k=1}^r a_{\vec{e}_j,k} - \sum_{k=1}^r a_{\vec{e}_i,k}}{b_0(\mu_j)-b_0(\mu_i)}, \]
provided $b_0(\mu_i) \neq b_0(\mu_j)$. Hence all the nearest neighbor recurrence coefficients are known for $|\vec{n}|=1$.

Suppose that we know all the nearest neighbor recurrence coefficients with multi-indices of length $\leq N$. Let $\vec{m}$ be a multi-index
of length $N+1$. In order to compute $a_{\vec{m},i}$ we choose a $j \neq i$ such that $m_i \geq 1$ and write $\vec{m} = \vec{n} + \vec{e}_j$,
where $\vec{n}$ is a multi-index of length $N$. 
Then from \eqref{PD3} we find that
\begin{equation}  \label{ar1}
  a_{\vec{m},i} =  a_{\vec{n}+\vec{e}_j,i} = a_{\vec{n},i} \frac{b_{\vec{n},j}-b_{\vec{n},i}}{b_{\vec{n}-\vec{e}_i,j}-b_{\vec{n}-\vec{e}_i,i}}.
\end{equation}
The coefficients on the right hand side have a multi-indices of length $N$ or $N-1$ and hence \eqref{ar1} allows us to compute
$a_{\vec{m},i}$ when $m_i=n_i \geq 1$. If $m_i=0$ then $a_{\vec{m},i}$ is the coefficient of the polynomial $P_{\vec{m}-\vec{e}_i}$ which is $0$, hence
we don't need this coefficient and we can set it equal to $0$. If $m_j=0$ for all $j \neq i$ then $\vec{m}=(N+1)\vec{e}_i$ and we have
$a_{(N+1)\vec{e}_i,i} = a_{N+1}^2(\mu_i)$.

If $i \neq j$ then the equations \eqref{PD1} and \eqref{PD2} are a linear system for the two unknowns $b_{\vec{n}+\vec{e}_i,j}$ and $b_{\vec{n}+\vec{e}_j,i}$. 
The solution is
\begin{equation}  \label{br1}
    b_{\vec{n}+\vec{e}_j,i} = b_{\vec{n},i} + \frac{\sum_{k=1}^r a_{\vec{n}+\vec{e}_j,k}- \sum_{k=1}^r a_{\vec{n}+\vec{e}_i,k}}
                              {b_{\vec{n},j}-b_{\vec{n},i}},
\end{equation}
provided that $b_{\vec{n},j} \neq b_{\vec{n},i}$. Hence if there exists $j \neq i$ with $m_j \geq 1$, then $\vec{m}=\vec{n}+\vec{e}_j$, and we can compute $b_{\vec{m},i}$ from \eqref{br1}.
If $m_j=0$ for all $j \neq i$ then $\vec{m}=(N+1)\vec{e}_i$ and $b_{(N+1)\vec{e}_i,i} = b_{N+1}(\mu_i)$.
\end{proof}

\end{document}